\documentclass{scrartcl}

\title{Grothendieck Pairs, Functors and\\Profinite Invariants}
\author{Daniel Echtler}

\KOMAoptions{ abstract=true }

\usepackage{polyglossia}
\setmainlanguage{english}

\usepackage{csquotes}
\usepackage[dvipsnames]{xcolor}
\usepackage{pifont}
\usepackage{array}
\usepackage{tabularx}
\usepackage{booktabs}
\usepackage{centernot}

\usepackage[maxnames = 5, giveninits = true]{biblatex}
\DeclareFieldFormat{eprint:urn}{URN\addcolon\space\href{https://nbn-resolving.org/#1}{\nolinkurl{#1}}}

\usepackage{amsmath}
\usepackage{amssymb}
\usepackage{amsthm}
\usepackage{mathtools}

\usepackage{tikz-cd}

\usepackage{thmtools}

\theoremstyle{plain}
  \declaretheorem{theorem}
  \declaretheorem[sibling=theorem]{proposition}
  
  \declaretheorem[sibling=theorem]{corollary}

  \declaretheorem{Theorem}

\theoremstyle{definition}
  \declaretheorem[sibling=theorem]{definition}
  \declaretheorem[sibling=theorem]{remark}
  \declaretheorem[sibling=theorem]{example}
  \declaretheorem[sibling=theorem]{question}

\newcommand\N{{\mathbb{N}}}
\newcommand\Z{{\mathbb{Z}}}
\newcommand\Q{{\mathbb{Q}}}
\newcommand\R{{\mathbb{R}}}

\newcommand\Group{{\operatorname{Group}}}
\newcommand\Abelian{{\operatorname{Ab}}}
\newcommand\Vect{{\operatorname{Vect}}}
\newcommand\Set{{\operatorname{Set}}}

\DeclareMathOperator{\id}{id}
\DeclareMathOperator{\Hom}{Hom}
\DeclareMathOperator{\Bohr}{Bohr}
\DeclareMathOperator{\coker}{coker}

\newcommand\blank{\ensuremath{\;\cdot\;}}
\newcommand\Ab[1]{\ensuremath{{{#1}_\mathrm{ab}}}}

\AtBeginDocument{
  \let\tmp\phi
  \let\phi\varphi
  \let\varphi\tmp
  
  \let\tmp\epsilon
  \let\epsilon\varepsilon
  \let\varepsilon\tmp
  
  \let\tmp\hat
  \let\hat\widehat

  \let\tmp\bar
  \let\bar\overline
  
}

\usepackage{hyperref}
\usepackage{microtype}

\newcaptionname{english}{\chapterautorefname}{Chapter}
\newcaptionname{english}{\sectionautorefname}{Section}
\newcaptionname{english}{\subsectionautorefname}{Section}

\begin{document}
  \maketitle

  \begin{abstract}\noindent
    As another variation of profinite properties, we define the notion of
    \emph{natural profinite invariants}: functorial invariants, inducing
    isomorphisms for every Grothendieck pair. We give a characterisation of when
    adjoint functors are natural profinite invariants and describe a strategy
    on how to disprove natural profinite invariance. By means of examples we
    show that the notions of profinite invariance and natural profinite
    invariances are logically independent.
  \end{abstract}

  \section{Introduction}
    A property $\mathcal{P}$ of groups is called a \emph{profinite} property
    if for any two finitely generated residually finite groups $\Gamma$ and 
    $\Lambda$ with $\hat\Gamma \cong \hat\Lambda$, the group $\Gamma$ has
    $\mathcal{P}$ if and only if $\Lambda$ has $\mathcal{P}$.
    % Similarly, an invariant of groups is called profinite if it agrees for any two finitely
    % generated residually finite groups with isomorphic profinite completions.
    While there are some positive examples of properties being profinite
    \cite{polycyclic_profinite, kionke_schesler, kammeyer_kionke_raimbault_sauer, coxeter_groups_profinite},
    more often than not they fail to be profinite \cite{aka_property_t, lubotzky_finiteness, kammeyer_sauer, kionke_schesler, bridson_fixed_point, FA_not_profinite, bc_not_profinite}.

    An embedding $\phi\colon \Lambda \to \Gamma$ of finitely generated
    residually finite groups is called a \emph{Grothendieck pair} if the
    induced morphism $\hat\phi\colon \hat\Lambda \to \hat\Gamma$ is an
    isomorphism. These pairs were first studied by Grothendieck \cite{grothendieck}.
    In his paper, Grothendieck raised the question wheter there are non-trivial
    Grothendieck pairs of finitely presented groups. The first \emph{finitely generated}
    example were given by Platonov and Tavgen' \cite{platonov_tavgen}, while the first finitely presented
    example was given by Bridson and Grunewald \cite{bridson_grunewald}.

    As a variation of a profinite property, a property of groups $\mathcal{P}$
    is said to be \emph{up Grothendeick} if for every Grothendieck pair
    $\Lambda \to \Gamma$ where $\Lambda$ satisfies $\mathcal{P}$, $\Gamma$
    satisfies $\mathcal{P}$ as well. Similarly, $\mathcal{P}$ is said to be
    \emph{down Grothendieck} if $\Gamma$ having $\mathcal{P}$ implies that $\Lambda$
    also does.

    As yet another modification we propose the notion of \emph{natural profinite invariants}.
    \begin{definition}[Natural profinite invariant]
      A functor $F\colon \Group \to \mathcal{D}$, from the category of groups
      to any category $\mathcal{D}$, is a \emph{natural profinite invariant} if
      for every Grothendieck pair $\phi\colon \Lambda \to \Gamma$ the induced
      morphism
      \begin{equation*}
        F(\phi)\colon F(\Lambda) \longrightarrow F(\Gamma)
      \end{equation*}
      is an isomorphism.
    \end{definition}

    In disguise, these were already considered by Grothendieck \cite{grothendieck} as well as
    Jaikin-Zapirain and Lubotzky \cite{jaikin_lubotzky}. In both papers, the authors consider
    the class $\mathcal{C}$ of groups $K$ such that
    \begin{align*}
      \phi^\ast\colon \Hom_{\Group}(\Gamma, K) &\longrightarrow \Hom_{\Group}(\Lambda, K) \\
                                             f &\longmapsto     f \circ \phi
    \end{align*}
    is a bijection for every Grothendieck pair $\phi\colon\Lambda \to \Gamma$.
    In other words, they studied for which groups $K$ the functor
    \begin{equation*}
      \Hom_\Group(\blank, K) \colon \Group \longrightarrow \Set^\mathrm{op}
    \end{equation*}
    is a natural profinite invariant. This culminated in the following theorem.
    \begin{theorem}[{\cite{grothendieck}, \cite[Theorem 1.1]{jaikin_lubotzky}}]
      \label{thm:grothendieck_class_C}
      \leavevmode
      \begin{enumerate}
        \item The class $\mathcal{C}$ is closed under commensurability, inverse
          limits and direct products, and contains all nilpotent groups.

        \item If $A$ is a commutative ring and $G$ is an affine group scheme
          of finite type over $A$, then $G(A) \in \mathcal{C}$.

        \item Compact Hausdorff groups are in $\mathcal{C}$.

        \item Let $\Lambda$ be a finitely generated residually finite group.
          If $\Lambda \in \mathcal{C}$, then $\Lambda$ is \emph{left Grothendieck rigid},
          i.e.\ every Grothendieck pair $\Lambda \to \Gamma$ is trivial.
      \end{enumerate}
    \end{theorem}

    In \autoref{sec:adjoint_functors} we will use the class $\mathcal{C}$ to
    give the following characterisation of natural profiniteness for 
    \emph{adjoint} functors.
    \begin{Theorem}[Adjoint functors]
      Let $\mathcal{D}$ be a locally small category and let $F\colon \Group \to \mathcal{D}$
      be left adjoint to $G\colon \mathcal{D} \to \Group$. Then $F$ is a
      natural profinite invariant if and only if~$G$ factors through~$\mathcal{C}$.
    \end{Theorem}

    In combination with \autoref{thm:grothendieck_class_C} this allows us to
    give even more examples of natural profinite invariants.
    First, we will give some easy examples of invariants that are (not) natural
    profinite invariants in \autoref{sec:examples}. Additionally we will
    describe an approach of how to disprove natural profiniteness in \autoref{sec:strategy}.

    The examples we examine are summarised in \autoref{tab:results}.
    \begin{table}
      % \begin{center}
        \def\cmark{\textcolor{ForestGreen}{\ding{51}}}
        \def\xmark{\textcolor{Red}{\ding{55}}}
        \def\exampleautorefname{Ex.}
        \def\propositionautorefname{Prop.}
        \def\corollaryautorefname{Cor.}
        \def\arraystretch{1.3}
        % \begin{tabular}{r @{${}\colon{}$} l >{\raggedleft}p{.8cm} l >{\raggedleft}p{.8cm} l}
        %   \multicolumn{2}{c}{\textbf{Functor}}                              & \multicolumn{2}{c}{\textbf{profinite}}        & \multicolumn{2}{c}{\textbf{natural profinite}}  \\\toprule
        %               $\hat{\blank}$ & $\Group \to \Group$                  & \cmark & (by definition)                      & \cmark & (by definition)                        \\
        %              $\Ab{(\blank)}$ & $\Group \to \Abelian$                & \cmark & \cite[Prop.\ 3.2]{reid_survey}       & \cmark & (\autoref{prop:ab_natural})            \\\midrule
        %                        $\id$ & $\Group \to \Group$                  & \xmark & (\autoref{ex:trivial})               & \xmark & (\autoref{ex:trivial})                 \\
        %          $H_b^2(\blank; \R)$ & $\Group \to \Vect_\R$                & \xmark & \cite{bc_not_profinite}              & \xmark & (\autoref{prop:BC_not_natural})        \\
        %   $H_{n \geq 2}(\blank; \Z)$ & $\Group \to \Abelian$                & \xmark & (\autoref{cor:H_n_not_profinite})    & \xmark & (\autoref{cor:H_n_not_profinite})      \\\midrule
        %      $\operatorname{Forget}$ & $\Group \to \Set$                    & \cmark & (\autoref{ex:trivial})               & \xmark & (\autoref{ex:trivial})                 \\\midrule
        %                      $\Bohr$ & $\Group \to \operatorname{CHGroup}$  & \xmark & \cite[Prop.\ 6.1]{jaikin_lubotzky}   & \cmark & (\autoref{cor:application_adjoint})
        % \end{tabular}
        \renewcommand\tabularxcolumn[1]{m{#1}}
        \newcolumntype{C}{>{\centering\arraybackslash}X}%
        \newcolumntype{M}{>{\centering\arraybackslash}m}%
        \begin{tabularx}{.98\textwidth}{r @{${}\colon{}$} l M{1.9cm} M{1.9cm} C}
          \multicolumn{2}{c}{\textbf{Functor}}                              & \textbf{Profinite?} & \textbf{Natural Profinite?} & \textbf{Reference(s)}                                       \\\toprule
                      $\hat{\blank}$ & $\Group \to \Group$                  & \cmark              & \cmark                      & By definition                                               \\
                     $\Ab{(\blank)}$ & $\Group \to \Abelian$                & \cmark              & \cmark                      & \cite[Prop.\ 3.2]{reid_survey}, \autoref{prop:ab_natural}   \\\midrule
                               $\id$ & $\Group \to \Group$                  & \xmark              & \xmark                      & \autoref{ex:trivial}                                        \\
                 $H_b^2(\blank; \R)$ & $\Group \to \Vect_\R$                & \xmark              & \xmark                      & \cite{bc_not_profinite}, \autoref{prop:BC_not_natural}      \\
          $H_{n \geq 2}(\blank; \Z)$ & $\Group \to \Abelian$                & \xmark              & \xmark                      & \autoref{cor:H_n_not_profinite}                             \\\midrule
             $\operatorname{Forget}$ & $\Group \to \Set$                    & \cmark              & \xmark                      & \autoref{ex:trivial}                                        \\\midrule
                             $\Bohr$ & $\Group \to \operatorname{CHGroup}$  & \xmark              & \cmark                      & \cite[Prop.\ 6.1]{jaikin_lubotzky}, \autoref{cor:application_adjoint}
        \end{tabularx}
      \caption{Summary of examples}
      \label{tab:results}
      % \end{center}
    \end{table}
    In particular, we obtain:
    \begin{Theorem}[Beyond logical dependence]
      The notions of profinite invariance and natural profinite invariance
      are logically independent.
    \end{Theorem}

    \begin{question}
      While not part of this note, it is an interesting question of whether
      additional assumptions restore a logicial dependence:

      Let $F\colon \Group \to \mathcal{D}$ be a functor. What are necessary/sufficient
      conditions such that natural profinite invariance of $F$ implies profinite
      invariance of $F$?
      Conversely, what are necessary/sufficient conditions such that profinite
      invariance of $F$ implies natural profinite invariance?
    \end{question}

  \subsection*{Acknowledgements}
    I would like to thank my advisor Holger Kammeyer for fruitful discussions
    and reading a first draft.

    This work was funded by the RTG \enquote{Algebro-Geometric Methods in
    Algebra, Arithmetic, and Topology} (DFG 284078965). Its content is part
    of the authors doctoral thesis \cite[Chapter 5]{echtler_phd}.

  \section{First (Non-)Examples}
  \label{sec:examples}
    Let us consider some first simple examples of functors and study their
    (natural) profiniteness.

    \begin{example}[trivial examples]
      \label{ex:trivial}
      \leavevmode
      \begin{enumerate}
        \item Take the identity functor $\id\colon \Group \to \Group$. Since
          there are non-isomorphic groups with isomorphic profinite completions,
          this is \emph{not} a profinite invariant. Morever, as there are
          non-trivial Grothendieck pairs, it \emph{not} a natural profinite
          invariant neither.

        \item By definition the profinite completion functor $\hat{\blank}\colon \Group \to \Group$
          is both a profinite and a natural profinite invariant.

        \item Consider the forgetful functor $\operatorname{Forget}\colon \Group \to \Set$. Using
          that there are non-trivial Grothendieck pairs we see that it is
          \emph{not} a natural profinite invariant. However, it turns out to
          be a profinite invariant:

          A finitely generated group is either finite, in which case it is
          profinitely rigid, or countably infinite. Thus any two finitely
          generated profinitely isomorphic groups have the same cardinality.
      \end{enumerate}
      The example of the forgetful functor showes that
      \begin{center}
        profinite invariance $\centernot\Longrightarrow$ natural profinite invariance.
      \end{center}
    \end{example}

    \begin{remark}[Localisation]
      \label{rmk:localisation}
      Let $S$ be the collection of all Grothendieck pairs and let $\Group[S^{-1}]$
      be the localisation of the category of groups at this class of morphisms.
      By construction the localisation functor $L_S\colon \Group \to \Group[S^{-1}]$
      is a natural profinite invariant. Moreover, it is \emph{universal} in the
      sense that a functor $F\colon \Group \to \mathcal{D}$ is a natural profinite
      invariant if and only if it factors (up to natural isomorphism) through the
      localisation functor.
      \begin{equation*}
        \begin{tikzcd}
          \Group \arrow[r, "F"]\arrow[d, "L_S"'] & \mathcal{D} \\
          \Group[S^{-1}] \arrow[ur, dashed]
        \end{tikzcd}
      \end{equation*}
    \end{remark}

    One of the standard examples of a profinite property is being abelian and
    the abelianisation \cite[Proposition 3.2]{reid_survey} \cite[5.2]{bridson_shadows}. It turns out that
    the abelianisation is a natural profinite invariant as well.

    \begin{proposition}[Abelianisation]
      \label{prop:ab_natural}
      Let $\phi\colon \Lambda \to \Gamma$ be a Grothendieck pair. Then the
      induced map on the abelianisation $\Ab\phi\colon \Ab\Lambda \to \Ab\Gamma$
      is an isomorphism.
    \end{proposition}
    \begin{proof}
      First we note that the universal properties of the profinite completion
      and the abelianisation can be used to show that the profinite completion
      of the abelianisation is naturally isomorphic to the (topological)
      abelianisation of the profinite completion. We simply write
      $\Ab{\hat{(\blank)}}$ for this composition of functors. This allows us to
      consider the commutative cube
      \begin{equation*}
        \begin{tikzcd}[sep=scriptsize]
          & \Ab{\hat\Lambda} \arrow[rr, "\Ab{\hat\phi}"]\arrow[from=dd, hook] && \Ab{\hat\Gamma} \\
          \hat\Lambda \arrow[ru]\arrow[rr, crossing over, near end, "\hat\phi"] && \hat\Gamma\arrow[ru] \\
          & \Ab\Lambda \arrow[rr, near start, "\Ab\phi"'] && \Ab\Gamma\arrow[uu, hook] \\
          \Lambda \arrow[ru]\arrow[rr, "\phi"']\arrow[uu, hook] && \Gamma \arrow[ru]\arrow[uu, crossing over, hook] \\
        \end{tikzcd}
      \end{equation*}
      where the vertical morphisms are the inclusion into the profinite
      completions and the \enquote{front-to-back} morphisms are the quotient
      maps to the abelianisation.

      Now by assumption $\phi\colon \Lambda \to \Gamma$ is a Grothendieck
      pair, whence $\hat\phi\colon \hat\Lambda \to \hat\Gamma$ and
      $\Ab{\hat\phi}\colon \Ab{\hat\Lambda} \to \Ab{\hat\Gamma}$ are
      isomorphism. Thus, in particular, $\Ab\phi\colon \Ab\Lambda \to \Ab\Gamma$
      is injective. Consider now the short exact sequence
      \begin{equation*}
        \begin{tikzcd}
          1 \arrow[r] & \Ab\Lambda \arrow[r, "\Ab\phi"] & \Ab\Gamma \arrow[r] & \coker(\Ab\phi) \arrow[r] & 1.
        \end{tikzcd}
      \end{equation*}
      As profinite completion is right-exact the induced sequence
      \begin{equation*}
        \begin{tikzcd}
          \Ab{\hat\Lambda} \arrow[r, "\Ab{\hat\phi}"] & \Ab{\hat\Gamma} \arrow[r] & \hat{\coker(\Ab\phi)} \arrow[r] & 1
        \end{tikzcd}
      \end{equation*}
      is exact as well. However, $\Ab{\hat\phi}$ was an isomorphism, so we obtain
      $\hat{\coker(\Ab\phi)} \cong 1.$ As $\coker(\Ab\phi)$ is a finitely
      generated abelian group, and thus residually finite, this implies
      $\coker(\Ab\phi) \cong 1$ as well. Hence $\Ab\phi\colon \Ab\Lambda \to \Ab\Gamma$
      is an isomorphism.
    \end{proof}

    In \autoref{cor:application_adjoint}, we will see a more systematic way of
    proving that the abelianisation is a natural profinite invariant.

    \begin{corollary}
      Let $A$ be an abelian group. Both $H_1(\blank; A)$ and $H^1(\blank; A)$ are
      natural profinite invariants.
    \end{corollary}
    \begin{proof}
      By the universal coefficient theorems for homology and cohomology, we have
      natural isomorphisms
      \begin{align*}
        H_1(\blank; A) &\cong H_1(\blank; \Z) \otimes_\Z A \\
        H^1(\blank; A) &\cong \Hom_\Z \bigl( H_1(\blank; \Z), A \bigr).
      \end{align*}
      Since $H_1(\blank; \Z) \cong \Ab{(\blank)}$ is a natural profinite
      invariant, we get the claim.
    \end{proof}

  \section{Natural Profiniteness of Adjoint Functors}
  \label{sec:adjoint_functors}

    In their paper, \citeauthor{jaikin_lubotzky} observed that the
    \emph{Bohr compactification} and the \emph{pro-algebraic completion}
    induce isomrophisms for every Grothendieck pair \cite[Corollary 1.3]{jaikin_lubotzky}. 
    Using the Yoneda lemma, we can abstract their proof and give a characterisation
    of the natural profinite invariance of \emph{adjoint} functors in terms of
    the class $\mathcal{C}$ studied by Grothendieck.

    \begin{theorem}[Nautral profinite invariance from adjointness]
      \label{thm:adjoint_functors}
      Let $\mathcal{D}$ be a locally small category and let $F\colon \Group \to \mathcal{D}$
      be left adjoint to $G\colon \mathcal{D} \to \Group$. Then $F$ is a natural
      profinite invariant if and only if for every $X \in \mathcal{D}$ we have
      $G(X) \in \mathcal{C}$, i.e., if $G$ factors over $\mathcal{C}$.
    \end{theorem}
    \begin{proof}
      Let $\phi\colon \Lambda \to \Gamma$ be a Grothendieck pair. By the adjunction
      of $F$ and $G$ we have a commutative diagram
      \begin{equation*}
        \begin{tikzcd}
          \Hom_\Group\bigl(\Gamma, G(\blank)\bigr) \arrow[r, "\phi^\ast"]\arrow[d, "\cong"'] &
          \Hom_\Group\bigl(\Lambda, G(\blank)\bigr) \arrow[d, "\cong"] \\
          \Hom_{\mathcal{D}}\bigl(F(\Gamma), \blank\bigr) \arrow[r, "F(\phi)^\ast"'] &
          \Hom_{\mathcal{D}}\bigl(F(\Lambda),\blank\bigr)
        \end{tikzcd}
      \end{equation*}
      whose morphisms are natural transformations.

      First, assume that for each $X \in \mathcal{D}$ we have $G(X) \in \mathcal{C}$.
      Then in particular
      \begin{equation*}
        \phi^\ast\colon \Hom_\Group\bigl(\Gamma, G(X)\bigr) \longrightarrow \Hom_\Group\bigl(\Lambda, G(X)\bigr)
      \end{equation*}
      is an isomorphism for every $X \in \mathcal{D}$. So the upper natural
      transformation is in fact a natural isomorphism. By commutativity, so is
      the lower natural transformation and using the Yoneda embedding, we obtain
      that $F(\phi)\colon F(\Lambda) \to F(\Gamma)$ is an isomorphism.
      As this holds for every Grothendieck pair, we obtain that $F$ is a natural
      profinite invariant.

      Conversely, assume that $F$ is a natural profinite invariant. Then
      $F(\phi)$ is an isomorphism, and thus $F(\phi)^\ast$ is a natural
      isomorphism. By commutativity, we get that $\phi^\ast$ is a natural
      isomorphism as well, i.e.\ for every $X \in \mathcal{D}$
      \begin{equation*}
        \phi^\ast\colon \Hom_\Group\bigl( \Gamma, G(X) \bigr) \longrightarrow \Hom_\Group\bigl( \Lambda, G(X) \bigr)
      \end{equation*}
      is an isomorphism. Again, this holds for every Grothendieck pair
      $\phi\colon \Lambda \to \Gamma$, and thus we obtain $G(X) \in \mathcal{C}$.
    \end{proof}

    \begin{corollary}[Examples of natural profinite invariants]
      \label{cor:application_adjoint}
      The following functors are natural profinite invariants:
      \begin{enumerate}
        \item the abelianisation $\Ab{(\blank)}\colon \Group \to \Abelian$,
        \item the Bohr compactification $\operatorname{Bohr}\colon \Group \to \operatorname{CHGroup}$, and
        \item the pro-algebraic completion $A\colon \Group \to \operatorname{pro-AlgGrp}$.
      \end{enumerate}
    \end{corollary}
    \begin{proof}
      All of these functors are left-adjoint to the corresponding forgetful
      functor. Moreover, by \autoref{thm:grothendieck_class_C} these
      right-adjoints factor over $\mathcal{C}$, and thus the above theorem gives 
      the natural profiniteness.
    \end{proof}

    \citeauthor{jaikin_lubotzky} moreover proved that the Bohr compactification
    is \emph{not} a profinite invariant \cite[Proposition 6.1]{jaikin_lubotzky}.
    Whence we can conclude:
    \begin{corollary}
      The Bohr compactification is a natural profinite invariant which is
      \emph{not} a profinite invariant, so
      \begin{center}
        natural profinite invariance $\centernot\Longrightarrow$ profinite invariance.
      \end{center}
    \end{corollary}

    \begin{remark}
      The fact that there is a natural profinite invariant $F\colon \Group \to \mathcal{D}$
      which is \emph{not} a profinite invariant allows us to give another such
      example: the localisation functor $L_S\colon \Group \to \Group[S^{-1}]$
      at the class of Grothendieck pairs (\autoref{rmk:localisation}).

      If $L_S$ was a profinite invariant, then the factorisation of $F$ through
      $L_S$ would give that $F$ is a profinite invariant as well. But this is
      not the case.

      \begin{equation*}
        \begin{tikzcd}
          \Group \arrow[r, "F"]\arrow[d, "L_S"'] & \mathcal{D} \\
          \Group[S^{-1}] \arrow[ur, dashed]
        \end{tikzcd}
      \end{equation*}
    \end{remark}

  \section{Constructing Grothendieck Pairs}
  \label{sec:strategy}

    One way of producing Grothendieck pairs is the approach \citeauthor{platonov_tavgen}
    took to produce the first examples of finitely generated Grothendieck pairs
    \cite{platonov_tavgen}.

    \begin{theorem}[{\cite[Theorem 5.1]{bridson_grunewald}}]
      Let $\phi\colon G \to Q$ be a surjective group homomorphism and let
      \begin{equation*}
        P = \bigl\{ (g,h) \in G \times G \mid \phi(g) = \phi(h) \bigr\} \leq G \times G
      \end{equation*}
      be the associated fiber product. If $G$ is finitely generated, $Q$ has no
      finite quotients, and $H_2(G; \Z) = 0$, then the inclusion $P \to G \times G$
      is a Grothendieck pair.
    \end{theorem}

    Platonov and Tavgen' applied this to the morphism $F_4 \to Q$, with
    \begin{equation*}
      Q = \langle x_1, x_2, x_3, x_4 \mid 
      x_2 x_1 x_2^{-1} = x_1^2,
      x_3 x_2 x_3^{-1} = x_2^2,
      x_4 x_3 x_4^{-1} = x_3^2,
      x_1 x_4 x_1^{-1} = x_4^2 \rangle
    \end{equation*}
    the Higman group.

    Bridson and Grunewald combined these arguments with a Rips-construction
    and the \enquote{1-2-3 Theorem} in order to obtain the first finitely
    presented Grothendieck pairs.

    \begin{theorem}[Wise--Rips construction \cite{rips_construction_residually_finite}]
      There is an algorithm that associates with every finite group presentation
      a short exact sequence of groups
      \begin{equation*}
        1 \longrightarrow N \longrightarrow G \longrightarrow Q \longrightarrow 1,
      \end{equation*}
      where $Q$ is the group with the given presentation, $N$ is generated by
      three elements, and the group $G$ is a torsion-free, residually finite,
      hyperbolic (and thus finitely presented) group.
    \end{theorem}

    Applying this again to the Higman group gives a short exact sequence whose
    fiber product is finitely presented. By the 1-2-3 Theorem, this fiber product
    is finitely presented and thus gives an example of a finitely presented
    Grothendieck pair.

    \begin{theorem}{\cite[1-2-3 Theorem]{fiber_product_123}}
      Let $1 \to N \to G \overset\phi\to Q \to 1$ be a short exact sequence and
      detote by $P$ the fiber product of $\phi$. If $N$ is finitely generated
      (type $F_1$), $G$ is finitely presented (type $F_2$), and $Q$ is of type
      $F_3$, then $P$ is finitely presented.
    \end{theorem}

    Essentially the same arguments as used by Platonov--Tavgen' prove the following.
    \begin{proposition}[{\cite[Lemma 4.4]{bridson_fixed_point}}]
      \label{prop:ses_grothendieck_pairs}
      Let $1 \to N \to G \to Q \to 1$ be a short exact sequence of finitely
      generated groups. If $\hat{Q} \cong 1$ and $H_2(Q; \Z) \cong 0$, then
      $N \to G$ is a Grothendieck pair.
    \end{proposition}

    In combination with the Wise--Rips construction, this allows us to find
    many short exact sequences whose left most map is a Grothendieck pair.

  \section{A Strategy for Disproving Natural Profiniteness}
    Assume we are given a functorial group invariant with a \enquote{good} way of
    discribing what happens when applied to a short exact sequence. This
    might be exactness, right-exactness, a long exact sequence, a five-term exact
    sequence/spectral sequence, etc. Then one can try to find a suitable
    \enquote{seed} group $Q$ with $\hat{Q} \cong 1$ and $H_2(Q; \Z)$ such that
    the Grothendieck pair obtained from the Wise--Rips construction allows 
    to disprove the natural profinite invariance of the given invariant.

    Of course here the difficultiy lies in finding a good seed group.
    Let us give two examples to illustrate this strategy: the second bounded
    cohomology and the higher homology groups.

    \begin{proposition}
      \label{prop:BC_not_natural}
      The second bounded cohomology $H_b^2(\blank; \R)$ is \emph{not} a natural
      profinite invariant.
    \end{proposition}
    \begin{proof}
      We again begin with the Higman group $Q$. Applying the Wise--Rips
      construction we obtain a short exact sequence
      \begin{equation*}
        \begin{tikzcd}
          1 \arrow[r] & N \arrow[r, "\iota"] & G \arrow[r] & Q \arrow[r] & 1
        \end{tikzcd}
      \end{equation*}
      of finitely generated groups with $N \to G$ a Grothendieck pair.

      This short exact sequence induces a five-term exact sequence
      \begin{equation*}
        \begin{tikzcd}[sep=scriptsize]
          1 \arrow[r] &
          H_b^2(Q; \R) \arrow[r] &
          H_b^2(G; \R) \arrow[r]\arrow[dr, "H_b^2(\iota)"'] &
          H_b^2(N; \R)^H \arrow[r]\arrow[d, hook] &
          H_b^3(Q; \R) \arrow[r] &
          H_b^3(G; \R) \\
          &&& H_b^2(N; \R)
        \end{tikzcd}
      \end{equation*}
      in bounded cohomology \cite[Theorem 2.21]{frigerio_bounded_cohomology}.
      Now, by construction, $Q$ is a non-trivial amalgam \cite{higman_group},
      so we have that $H_b^2(Q; \R)$ is non-trivial (in fact, it is infinite
      dimensional) \cite{fujiwara_amalgam}. Hence, by exactness, the induced
      map $H_b^2(\iota)\colon H_b^2(G; \R) \to H_b^2(N; \R)^H \hookrightarrow H_b^2(N; \R)$
      is \emph{not} injective.
    \end{proof}

    \begin{remark}
      In an earlier paper joint with Holger Kammeyer \cite{bc_not_profinite}
      we disproved the profiniteness of the second bounded cohomology. The
      above proof does not provide any new examples of this failure: 

      By construction, $G$ is a non-elementary hyperbolic group. Moreover,
      as $N \to G$ is a (non-trivial) Grothendieck pair, we have that $N$ is
      an infinite normal subgroup of $G$. Hence both $G$ and $N$ are
      \emph{acylindrically hyperbolic groups}
      \cite[Corollary 1.5]{osin_acylindrically_hyperbolic}, and so both
      $H_b^2(G; \R)$ and $H_b^2(H; \R)$ are infinite dimensional
      \cite[Theorem A]{hamenstädt_hyperbolic}.
    \end{remark}

    Another example where whe can apply a similar strategy is the second
    homology group. Solving a problem posed by Baumslag, \citeauthor{bridson_reid}
    showed the following.
    \begin{theorem}[{\cite[Theorem 6.1]{bridson_reid}}]
      There exists a pair of finitely generated residually torsion-free
      nilpotent groups $N \hookrightarrow \Gamma$ that have the same nilpotent
      genus and the same profinite completion, but $\Gamma$ is finitely 
      presented while $H_2(N; \Q)$ is inifinite-dimensional.
    \end{theorem}

    Let us prove the following slight variation of this statement which is
    essentially already contained in their construction.
    \begin{proposition}
      There exists a Grothendieck pair $N \to G$ of finitely generated residually
      finite groups, where $G$ is hyperbolic and $H_2(N; \Z)$ is \emph{not}
      finitely generated.
    \end{proposition}
    \begin{proof}
      Bridson and Reid constructed a group $Q$ that is finitely presented,
      super-perfect, i.e.\ $H_1(Q; \Z) \cong H_2(Q; \Z) \cong 0$, has no finite
      quotients and has infinite dimensional $H_3(Q; \Q)$ \cite[Proposition 6.5]{bridson_reid}.
      Using the universal coefficient theorem, we get in particular that
      $H_3(Q; \Z)$ is not finitely generated.

      Using this group $Q$ as seed group we again obtain a short exact sequence
      \begin{equation*}
        \begin{tikzcd}
          1 \arrow[r] & N \arrow[r, "\iota"] & G \arrow[r] & Q \arrow[r] & 1
        \end{tikzcd}
      \end{equation*}
      of finitely generated groups, with $G$ hyperbolic and $\iota\colon N \to G$
      a Grothendieck pair. Applying the Hochschild--Serre spectral sequence
      to this short exact sequence, we obtain
      \begin{equation*}
        E^2_{p,q} \cong H_p\bigl( Q; H_q(N; \Z) \bigr) \Longrightarrow H_{p+q}( G; \Z).
      \end{equation*}
      As $N$ is finitely generated, $H_1(N; \Z) \cong \Ab{N}$ is a finitely
      generated abelian group. Since $Q$ has no finite quotients, the action
      of $Q$ on $H_1(N; \Z)$ has to be trivial. Thus we can use the universal
      coefficient theorem to obtain
      \begin{equation*}
        H_1\bigl( Q; H_1(N; \Z) \bigr) \cong \underbrace{H_1(Q; \Z)}_\text{$\cong 0$} \otimes H_1(N; \Z) \cong 0,
      \end{equation*}
      as well as
      \begin{equation*}
        H_2\bigl( Q; H_1(N; \Z) \bigr) \cong \underbrace{H_2(Q; \Z)}_\text{$\cong 0$} \otimes H_1(N; \Z) \oplus \operatorname{Tor}_1^\Z\bigl( \underbrace{H_1(Q; \Z)}_\text{$\cong 0$}, H_1(N; \Z) \bigr) \cong 0.
      \end{equation*}

      Hence the second page of the spectral sequence looks as follows
      \begin{center}
        \begin{tikzpicture}[scale=1.4]
          \draw[lightgray, ->] (-.4,0) -- (4.3,0) node[below]{$p$};
          \draw[lightgray, ->] (0,-.4) -- (0,3.3) node[left]{$q$};

          \node at (0,0) {$\Z$};
          \node at (1,0) {$0$};
          \node at (2,0) {$0$};
          \node at (3,0) {$H_3(Q; \Z)$};

          \node at (0, 1) {$H_1(N; \Z)$};
          \node at (0, 2) {$H_2(N; \Z)_Q$};

          \node at (1, 1) {$0$};
          \node at (2, 1) {$0$};

          \draw[->, shorten <=.65cm, shorten >=.25cm] (3, 0) -- node[below left]{$d^2$} (1,1);
          \draw[->, shorten <=.25cm, shorten >=.65cm] (2, 1) -- node[above right]{$d^2$} (0,2);

          \foreach \x/\y in {4/0, 3/1, 4/1, 1/2, 2/2, 3/2, 4/2, 0/3, 1/3, 2/3, 3/3, 4/3}{
            \draw[fill, gray] (\x,\y) circle (1pt);
          }

          \node at (2,-.5) {$E^2$};
        \end{tikzpicture}
      \end{center}
      and these modules survive on the third page, where the differential
      is part of the exact sequence 
      \begin{equation*}
        \begin{tikzcd}
          1 \arrow[r] &
          \ker d^3_{3,0} \arrow[r] & 
          H_3(Q; \Z) \arrow[r, "d^3_{3,0}"] & 
          H_2(N; \Z)_Q \arrow[r] & 
          H_2(G; \Z) \arrow[r] & 
          1.
        \end{tikzcd}
      \end{equation*}
      Moreover, the convergence towards $H_{p+q}(G; \Z)$ gives
      \begin{equation*}
        \begin{tikzcd}
          H_3(G; \Z) \arrow[r] &
          E^\infty_{3,0} \cong \ker d^3_{3,0} \arrow[r] &
          1.
        \end{tikzcd}
      \end{equation*}
      Combining these two sequences yields
      \begin{equation*}
        \begin{tikzcd}
          H_3(G; \Z) \arrow[r] &
          H_3(Q; \Z) \arrow[r] &
          H_2(N; \Z)_Q \arrow[r] &
          H_2(G; \Z) \arrow[r] &
          1.
        \end{tikzcd}
      \end{equation*}

      Now $H_3(Q; \Z)$ is \emph{not} finitely generated, while $H_3(G; \Z)$ is
      (as $G$ is hyperbolic and thus of type $F_\infty$). Thus from the exactness
      of the sequence we can deduce that $H_2(N; \Z)_Q$, and thus in particular
      $H_2(N; \Z)$, is \emph{not} finitely generated.
    \end{proof}

    Since $G$ is hyperbolic we have that $H_2(G; \Z)$ is finitely generated, 
    so in particular the induces morphism $H_2(N; \Z) \to H_2(G; \Z)$ is not an 
    isomorphism. Using a kind of dimension shifting argument this result
    can be lifted to the higher homology groups.

    \begin{corollary}
      \label{cor:H_n_not_profinite}
      For every $n \in \N_{\geq 2}$ the $n$-th homology group with integer
      coefficients $H_n(\blank; \Z)$ is neither a profinite invariant, nor a 
      natural profinite invariant.
    \end{corollary}
    In fact we will prove that for every $n \geq 2$ there is a Grothendieck pair
    $N_n \to G_n$ such that $H_n(N_n; \Z)$ is \emph{not} finitely generated,
    while $G_n$ is of type $F_\infty$. Then $H_n(N_n; \Z)$ and $H_n(G_n; \Z)$
    are not abstractly isomorphic and so in particular $H_n(N_n; \Z) \to H_n(G_n; \Z)$
    is not an isomorphism.
    \begin{proof}
      For $n=2$ we can take $N_2 \to G_2$ to be the Grothendieck pair given
      by the above proposition. From here on we proceed by induction on $n$.

      Suppose we have already constructed a Grothendeick pair $\iota_n\colon N_n \to G_n$
      with $H_n(N_n; \Z)$ \emph{not} finitely generated and $G_n$ of type $F_\infty$.
      We now consider $\iota_n \times \id_\Z \colon N_n \times \Z \to G_n \times \Z$.
      As taking profinite completions is compatible with products, this is
      still a Grothendieck pair. Moreover, using the Künneth theorem, we see
      that $G_n \times \Z$ is still of type $F_\infty$ and
      \begin{equation*}
        H_{n+1}(N_n \times \Z; \Z) \cong H_{n+1}(N_n; \Z) \oplus H_n(N_n; \Z).
      \end{equation*}
      As $H_n(N_n; \Z)$ is \emph{not} finitely generated we obtain that
      $H_{n+1}(N_n \times \Z; \Z)$ is \emph{not} finitely generated as well.
      Hence $N_n \times \Z \to G_n \times \Z$ is the desired Grothendieck pair.
    \end{proof}
  \printbibliography
\end{document}